%% file: FreegpsDemonstrative_arXiv2.tex
\documentclass[10pt,a4paper]{article}
\usepackage[utf8]{inputenc}
\usepackage{amsmath}
\usepackage{amsfonts}
\usepackage{amssymb}
\usepackage{mathrsfs}
\usepackage{graphicx}
\usepackage{amsthm}
\usepackage[margin=2cm]{caption}
\usepackage{subcaption}
\usepackage{floatrow}
\usepackage{sidecap}
\usepackage{tikz}
\usepackage[margin=1.25in]{geometry}

\title{A dynamical definition of f.g. virtually free groups}
\author{Daniel Bennett and Collin Bleak}

\newtheorem{theorem}{Theorem}
\newtheorem{conjecture}[theorem]{Conjecture}
\newtheorem{corollary}[theorem]{Corollary}
\newtheorem{lemma}[theorem]{Lemma}
\newtheorem{definition}[theorem]{Definition}
\newtheorem{prop}[theorem]{Proposition}

\newtheorem{cor}[theorem]{Corollary}

\newcommand{\cocf}{co\mathcal{CF}}
\newcommand{\cc}{\mathfrak{C}_2}
\newcommand{\Z}{\mathbb{Z}}
\newcommand{\Q}{\mathbb{Q}}
\newcommand{\seteq}{:=}

\begin{document}

\maketitle
\abstract{
We show that the class of finitely generated virtually free groups is precisely the class of demonstrable subgroups for R. Thompson's group $V$.   The class of demonstrable groups for $V$ consists of all groups which can embed into $V$ with a natural dynamical behaviour in their induced actions on the Cantor space $\cc \seteq \left\{0,1\right\}^\omega$.  There are also connections with formal language theory, as the class of groups with context-free word problem is also the class of finitely generated virtually free groups, while R. Thompson's group $V$ is a candidate as a universal $\cocf$ group by Lehnert's conjecture, corresponding to the class of groups with context free co-word problem (as introduced by Holt, Rees, R\"over, and Thomas).    Our main result answers a question of Berns-Zieze, Fry, Gillings, Hoganson, and Matthews, and separately of Bleak and Salazar-D\'iaz, and it fits into the larger exploration of the class of $\cocf$ groups as it shows that all four of the known closure properties of the class of $\cocf$ groups hold for the set of finitely generated subgroups of $V.$}
\section{Introduction}
There is a long history  of determining classes of groups based on some external criteria.  In the case of formal language theory, we mention in particular two classical results along these lines.  Anisimov's theorem \cite{anisimov} that the class of finitely generated groups with \emph{word problem} a regular language is precisely the class of finite groups, and Muller and Schupp's result \cite{mullerschupp83, mullerschupp85} that the class of finitely generated groups with word problem a context free language is precisely the class of finitely generated virtually free groups.

Holt, Rees, R\"over, and Thomas introduce in \cite{HRRT} the class $\cocf$ of co-context free groups.  The $\cocf$ groups are a natural generalisation of the class of groups with context free word problem, and so one can hope for another classification of these groups as some particular natural class of groups.  Presently, there is a conjecture of Lehnert \cite{lehnertschweitzer07,BMN13} that the class $\cocf$ is precisely the class of finitely generated groups that can embed in R. Thompson's group $V$.

The R. Thompson groups $F< T< V$ have been known to be important essentially since their introduction in the mid 1960's by Richard Thompson \cite{cfp,thompson}.  However, while they are fundamental objects, they remain quite mysterious in many ways.  Specific to $V$, we can highlight various investigations \cite{birget2004,lehnertschweitzer07,fryetal,olgaConjugacy,centralisersVn,collinolga,dehornoy}, which list is by no means comprehensive.  

One view of the group $V$, amongst others,  is as a group of homeomorphisms of the standard ``deleted middle thirds'' Cantor space $\cc\seteq \left\{0,1\right\}^\omega\!\!.$   Recently, it has become clear \cite{collinmartyn} that yet another view of $V$ is natural and interesting.  In this view, $V$ can be thought of as an infinite generalisation of the finite alternating groups (our most basic examples of finite simple groups).  From this view, $V$ can be seen as a permutation group of ``even'' permutations ``acting'' on a poset structure of some partitions of Cantor space.  This, perhaps, offers a new insight into why $V$ arose as one of the first examples of finitely presented infinite simple groups.

In any case, in \cite{collinolga} a class of subgroups of $V$ were singled out as being particularly important.  These are subgroups of $V$ with induced actions on Cantor space that are naturally ``Geometric''; the \emph{demonstrative subgroups} of $V$.  Groups (in general) which admit embeddings into $V$ as demonstrative subgroups of $V$ are known as \emph{demonstrable groups for $V$}.   We note in passing that Question 3 of \cite{collinolga} requests a classification of these groups, which we will answer here in the finitely generated case.

Thus, the main result of this paper from the point of view of the development of the dynamic theory of $V$ is the following.
\begin{theorem}\label{thm:mainTheorem}Let $G$ be a finitely generated group.
Then, $G$ is a demonstrable group for $V$ if and only if $G$ is a virtually free group.
\end{theorem}
In particular, as we noted before, this is precisely the class of context free groups.  

In passing, we comment that $\Q\backslash\Z$ is a known demonstrable subgroup for $V$ (and even, for $T$, see \cite{BKM}), so we need the ``finitely generated'' criterion.

As mentioned above, Lehnert in his dissertation \cite{LehnertDissertation} conjectures that a group $QAut(\mathscr{T}_{2,c})$ is a universal $\cocf$ group; that a group is a $\cocf$ group if and only if it embeds as a finitely generated subgroup of $QAut(\mathscr{T}_{2,c})$.  Lehnert also shows that $V$ embeds in $QAut(\mathscr{T}_{2,c})$.  Lehnert and Schweitzer show by a separate argument in \cite{lehnertschweitzer07} that $V$ is a $\cocf$ group.  In response to the work in \cite{collinolga}, Lehnert and Schweitzer \cite{lehnertSchweitzerRequest} asked those authors whether $V$ and $QAut(\mathscr{T}_{2,c})$ are bi-embeddable.  In \cite{BMN13}, this question is answered affirmatively.  Consequently, one can restate Lehnert's conjecture as the following.

\begin{conjecture}[Lehnert] R. Thompson's group $V$ is a universal $\cocf$ group.
\end{conjecture}

The paper \cite{fryetal} investigates a class of groups which could provide possible counter-examples to Lehnert's conjecture.  We are not certain at this time whether or not those groups can embed in $V$, but part of the motivation for the work here was the request in \cite{fryetal} to embed the virtually free groups into $V$ as demonstrative subgroups of $V$.  Indeed, the authors of \cite{fryetal} ask in Question 1 whether non-abelian free groups can embed as demonstrative subgroups of $V$ (from which it would follow that all countable virtually free groups could embed as demonstrative subgroups).  Thus we answer that question here in the affirmative.

Finally, concerning the main result of this article, we note that in \cite{HRRT} the class of $\cocf$ groups is shown to be closed under four operations: passing to finitely generated subgroup, or to finite index over-group, or to a finite direct product of $\cocf$ groups, or finally to a restricted wreath product of a $\cocf$ group with a context free top group.  It was already known that the finitely generated subgroups of $V$ are closed under the first three operations (see \cite{roverthesis,collinolga}), but the last operation remained a mystery.   However, it is shown in \cite{collinolga} that the subgroups of $V$ are closed under passage to a restricted wreath product of a subgroup of $V$ with a demonstrable group for $V$ (for top group). Thus, a corollary of the main theorem of this paper is that the finitely generated subgroups of $V$ are closed under this operation as well.  

\begin{corollary}\label{cor:fourProps}
The finitely generated subgroups of $V$ are closed under the following four properties:
\begin{enumerate}
\item passing to finitely generated subgroups, 
\item passing to finite index over-groups,
\item taking finite direct products, and
\item taking wreath products with any $\mathcal{CF}$ top group.
\end{enumerate}
\end{corollary}

In fact, we have a stronger result.

\begin{theorem}\label{wreath-embeddings}
Let $G$ be a subgroup of $V$, and let $T$ be any countable virtually free group.  Then, $G\wr T$ embeds as a subgroup of $V$.
\end{theorem}

In particular, the finitely generated subgroups of $V$ are closed under the four known closure properties of the $\cocf$ groups.

\subsection{Languages and groups}

We now introduce classes of groups that are defined by properties found in Formal Language Theory. We begin by defining a formal language.

\begin{definition}
Let $\Sigma$ be a finite set of elements which we will call an \textit{alphabet}. The \textit{free monoid} of $\Sigma$, denoted $\Sigma^\ast$, is the set of all finite strings made from the elements of $\Sigma$, under the operation of concatenation. We call any subset of $\Sigma^\ast$ a \emph{language over the alphabet $\Sigma$}.
\end{definition}


Many formal languages can be defined by mathematical machines called \textit{automata}. An automaton is simple theoretical computation device that reads a string of symbols from a given alphabet and decides whether to accept that string or reject it. We will later define more formally what it means for automata to accept or reject a string. The set of all strings that are accepted by an automaton make a formal language.

\begin{definition}[Finite Automata]
A finite automaton, is a \textit{5-tuple}, $M=(\Sigma,Q,\delta,q_0,F)$, where $\Sigma$ is a finite set called the alphabet, $Q$ is a finite set of \textit{states}, $\delta$ is the transition function between the states, $q_0$ is a \emph{start state} in $Q$, and $F\subset Q$ is called the set of \textit{accept states}. The transition function is defined as $\delta:Q \times \Sigma \to Q$.
\end{definition}

Let $M$ be a finite automaton as above and let $w=w_1w_2\ldots w_n\in \Sigma^\ast$ be a word in the alphabet $\Sigma$.  We say \emph{$M$ runs on $w$ } (or \emph{computes} $w$) when we carry out the following process.  We assume we are ``in'' the initial state $q_0$.  Then we ``read'' the first letter $w_1$ of $w$, and transition to a new state $q=\delta(q_0,w_1)$.  We then read the next letter ($w_2$) and transition to the next state $\delta(q,w_2)$.  We continue this process until we have read the last letter $w_n$ and moved to a state $z$.  If $z\in F$ we say that \emph{$M$ accepts $w$}.  Otherwise we say that \emph{$M$ rejects $w$}.

Then, the language of words in $\Sigma^\ast$ which are accepted by $M$ is called the \emph{language accepted by the automaton $M$}.

Now, a language $\mathcal{L}\subset \Sigma^*$ is a \emph{regular language} if and only if there is a finite automaton $M$ so that $\mathcal{L}$ is precisely the language accepted by $M$.

\begin{definition}
Let $G$ be a group such that $G = \langle \Sigma | R \rangle$ is a finite presentation for $G$ and $\Sigma$ is closed under taking inverses. Then there exits a homomorphism $\theta: \Sigma^\ast \to G$ from the free monoid $\Sigma^\ast$ to the group $G$, defined by $\theta(w)=(w)_R$ for all $w\in \Sigma^\ast$ where $(w)_R$ is the element in $G$ that $w$ represents. We define the \textit{word problem} of $G$ to be $$W(G) = \{w\in \Sigma^\ast | \theta(w)=id \}$$ where $id$ is the identity element in $G$. We say that the \textit{co-word} problem of $G$, or $coW(G)$, is the complement of $W(G)$, i.e. the set of elements in $\Sigma^\ast$ that do not evaluate to the identity under the relations $R$.
\end{definition}

We are now ready to state our first historical motivating theorem, given in \cite{anisimov}, using slightly different language than in the original paper.
\begin{theorem}[Anisimov]
Let $G=\langle \Sigma\mid R\rangle$ be a finitely generated group.  $G$ is finite if and only if $W(G)$ is a regular language.
\end{theorem}

There are many different types of automata, each one producing a different sort of formal language. We will be mostly interested in \textit{push down automata} which give rise to \textit{context free languages}.



\begin{definition}[Push-down Automata]
\label{PDAdef}
A push-down automaton, or \textit{PDA}, is a \textit{6-tuple}, $M=(\Sigma,\Gamma,Q,\delta,q_0,F)$, where $\Sigma$, $\Gamma$ and $Q$ are finite sets (which we call the \textit{input} alphabet of $M$, the \textit{stack} alphabet of $M$, and the set of states of $M$, respectively), $\delta$ is a transition relation, $q_0\in Q$ is the start state, and $F\subset Q$ the set of accept states. The transition relation is defined as $\delta:Q \times \Sigma \times \Gamma^\ast \to Q \times \Gamma^\ast$ where $\Gamma^\ast$ is the set of all finite strings over the alphabet $\Gamma$.
\end{definition}

The \textit{stack} is an external memory device which uses it's own alphabet $\Gamma$. It is a finite ordered set of elements from $\Gamma$, $(\gamma_0,\gamma_1,\gamma_2,\ldots,\gamma_{n-1},\gamma_n)$, $n\ge 0$, that the $PDA$ has limited access to. In a $PDA$ the transition relation uses a string of elements from the top of the stack, $\gamma_k\ldots\gamma_n$ for some $k\ge 1$ as one of it's arguments.


A \textit{PDA} computes as follows. We begin with a string of symbols in $\Sigma$, called the input string, and at the state $q_0$ in the automaton. Initially the stack is empty. We read the first of the symbols in our input string. The transition function then uses the current state, the symbol that we read and the symbol on top of the stack (if any) to determine which state to move to next. After reading in the first input symbol we move to the state determined by $\delta$ and progess the reader to the second symbol in our input string. As $\delta$ is a relation, and not necessarily a function, It is possible that $\delta$ gives more than one option, in this case we have a choice as to which state we move to, this type of automaton is called \textit{non-determinitistic}.  Additionally the transition relation also deletes whatever it reads on the stack and replaces it with a new string of symbols from the stack alphabet. It is also possible for the automaton to change states even when no input symbol has been read, these are called $\epsilon-moves$ and are determined solely by the current state and the stack.

Given an input string $w \in \Sigma^\ast$ and a \textit{PDA} $M$, we denote by $\mathcal{P}(w)$ the set of all legal paths in $M$ that are determined by $w$. If our automaton is non-deterministic then the set $\mathcal{P}(w)$ could be larger than one. If there exists some path $p \in \mathcal{P}(w)$ such that $p$ ends at a state $q_f \in F$,  the set of final states, then we say the automaton \textit{accepts} the string $w$. If none of the paths in $\mathcal{P}(w)$ end at a state in $F$ then we say the automaton \textit{rejects} the string $w$. The set of all words in $\Sigma^\ast$ that are accepted by the automata $M$ is a context free language.

Note that we could equivalently have changed our acceptance criteria to only accept a string $w \in \Sigma^\ast$ if there exists a legal path such that the stack finishes empty. For every language that is created by an automaton that accepts via final state we can always find an automaton that accepts the same language by the empty stack criteria, and vice versa. The proof of this is given by Theorem 5.1 and Theorem 5.2 in \cite{hopcroftullman}.

\begin{definition}
Let $G$ be a finitely generated group such that $W(G)$ is a context free language. Then we call $G$ a context free group and say that $G$ is $\mathcal{CF}$. Equivalently, if $coW(G)$ is a context free language then we say $G$ is $\cocf$.
\end{definition}

It is shown in \cite{HRRT} that the property of being $\cocf$ is independent of the choice of finite generating set.

We note that every $\mathcal{CF}$ group is also a $\cocf$ group, but the converse is not true. The $\mathcal{CF}$ groups are completely classified by Muller and Schupp in \cite{mullerschupp83} and \cite{mullerschupp85} to be exactly the finitely generated virtually free groups.   It is not too hard to see that the word problem of a finitely generated virtually free group is actually a \emph{deterministic} context free language (see \cite{diekertweiss}).  As the complement of a deterministic context free language is also a deterministic context free language, we see that such a group has its set of co-words also forming a deterministic context free language.  However, the complement of a non-deterministic context free language need not be context free, hence the class of $\cocf$ groups is broader than the class of $\mathcal{CF}$ groups (as is witnessed, for instance, by R. Thompson's group $V$, which is not virtually free).

\subsection{Cantor space $\cc$}

As mentioned before R. Thomspon's Group $V$ is a group of automorphisms of the Cantor set, $\cc$. We therefore begin with a brief discussion to define the notation and language regarding the Cantor set that we will use throughout this note. We follow closely the description given in \cite{collinolga} but also make mention of other conventions as required.

Let $X=\{0,1\}$ which we will call our \textit{alphabet} and define $X^\ast$ to be the set of all finite strings over $X$. We introduce $\mathcal{T}_2$, the infinite rooted binary tree, as a non-directed graph that has vertex set $X^\ast$. We define edges in $\mathcal{T}_2$ as follows. If $u,v \in X^\ast$ then there exists an edge between $u$ and $v$ if and only if $u=vx$ or $v=uy$, where $x,y \in X$. We define the root of $\mathcal{T}_2$ to be empty word $\epsilon \in X^\ast$. A vertex $n$ in $\mathcal{T}_2$ will also be referred to as a \textit{node} of $\mathcal{T}_2$ and we define the \textit{address} of a node $n$ to be the path in $X^\ast$ from the root of $\mathcal{T}_2$ to $n$. An \textit{infinite descending path} in $\mathcal{T}_2$ is an infinite string of $0$'s and $1$'s. The boundary of $\mathcal{T}_2$ is the set of all such infinite descending paths which we denote by $\{0,1\}^\omega$. We say two paths are near to each other if they share a long common prefix. Consider the paths $p_{n_1}$ and $p_{n_2}$ in Figure \ref{treepaths} which start from the root $\epsilon$ and end at the vertices $n_1$ and $n_2$ respectively. Notice that from the root to the vertex $k$ the two paths share the same route through $\mathcal{T}_2$, we would therefore say that $p_{n_1}$ and $p_{n_2}$ share a prefix of length $k$. This induces a topology on the boundary of the tree that is  equivalent to the product topology for $\{0,1\}^\omega= \cc$.

\begin{figure}[h] 
\centering \def\svgwidth{180pt} 
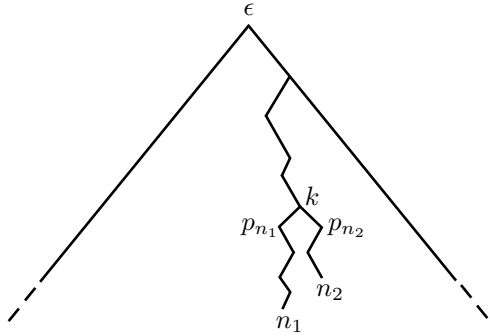 
\caption{Two paths in the infinite binary tree $\mathcal{T}_2$ with a common prefix of length $k$. Notice how we draw the tree with the root at the top.}
\label{treepaths} 
\end{figure}

Suppose $x \in\{0,1\}^\omega=\cc$, then we say that $x$ \textit{underlies} the vertex $w \in X^\ast$ of $\mathcal{T}_2$ if there exists $\hat{x} \in \{0,1\}^\omega$ such that $x=w\hat{x}$, or in other words $w$ is a prefix of the infinite string defining $x$. The set of all such strings in $\{0,1\}^\omega$ that underly the vertex $w$ in $\mathcal{T}_2$ is a clopen subset of $\{0,1\}^\omega$ which we will denote by $[w]$, and is in fact a basic open set in the product topology on $\{0,1\}^\omega$. Notice that $\cc \cong [w]$. In \cite{collinolga} the set $[w]$ was denoted by $\mathfrak{C}_w$.

\subsection {R. Thompson's group $V$, barriers, and prefix replacements} 

We now introduce some terminology used in \cite{coindexholtrover} that will help us to define Thompson's group $V$. We first define what Holt and R\"{o}ver call a \textit{barrier}. Let $\mathcal{B}$ be a finite set of finite strings in $X^\ast$ such that every element in $\{0,1\}^\omega \cong \mathfrak{C}$ has a unique string in $\mathcal{B}$ as a prefix. We call $\mathcal{B}$ a barrier. For example $\mathcal{B}$ could be the set $\{0,100,101,11\}$ but not the sets $\{0,100,11\}$ or $\{0,100,101,11,110\}$.

A barrier is equivalent to a complete finite antichain on the poset $X^\ast$ under the relation $x \le y \Leftrightarrow x$ is a prefix of $y$, as introducted by Birget \cite{birget2004}.

A \textit{prefix replacement} is a triple $g=(\mathcal{D},\mathcal{R},\sigma)$ where $\mathcal{D}$ and $\mathcal{R}$ are barriers and $\sigma$ is a bijection between them. This induces an action on $\mathfrak{C}_2$ which we call a \textit{prefix replacement map}, or \textbf{prm} for short (note that in \cite{coindexholtrover} this was called a prefix replacement \textit{permutation}). If $g$ is a prefix replacement then the \textbf{prm} induced from $g$ acts on $w \in \mathfrak{C}_2$ by replacing the prefix $d \in \mathcal{D}$ by the prefix $d^\sigma \in \mathcal{R}$. Note that there are many different prefix replacements which will give the same \textbf{prm}. We then define Thomspon's Group $V$ as the set of all \textbf{prm}'s on $\{0,1\}^\omega$, under composition.

We can draw barriers as finite, rooted subtrees of the infinite rooted binary tree $\mathcal{T}_2$. If $T$ is a subtree of $\mathcal{T}_2$ then we draw $T$ with the root at the top and the tree drawn ``downwards" away from the root. Given a vertex $u$ of $T$, we draw the child $u0$ of $u$ to the left of $u$ and the child $u1$ to the right. The vertices of $T$ that do not have any children we call \textit{leaves}. We can then describe unique paths through $T$ from the root to the leaves by elements of $X^\ast$, where a $0$ means travel down the next left hand edge, and $1$ means travel down the next right hand edge. A barrier will then define a finite binary tree where each element of the barrier will describe a path from the root to a unique leaf. We can now use this construction to represent elements of $V$ as pairs of binary trees. If the prefix replacement $g=(\mathcal{D},\mathcal{R},\sigma)$ induces an element $v \in V$, then the two barriers $\mathcal{D}$ and $\mathcal{R}$ define two binary trees and the bijection $\sigma$ is represented by a numerical labeling on the leaves of the trees. The tree pair representation of an example element $g \in V$, that will appear again later on the in paper, is given in Figure \ref{fig:treepair}. We typically call the tree on the left the \textit{domain} tree and the tree on the right the \textit{range} tree. In our example the barriers that define the domain and range trees are $\{0,100,101,11\}$ and $\{0,10,110,111\}$ respectively. The numbering of the leaves represents the bijection between the two barriers, where the leaf labelled $``1"$ in the domain tree is taken to the leaf labelled $``1"$ in the range tree and so on. 


\begin{figure}[H]
		\centering \def\svgwidth{250pt}
		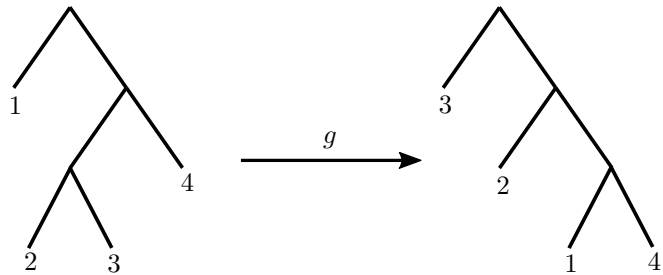
		\caption{The tree pair representation of an element $g$ of Thompson's group $V$}
		\label{fig:treepair}
\end{figure}

\subsection{On ``Geometrical'' embeddings}
With the growth of Geometric Group Theory as an independent field, there have been many types of group actions and embeddings defined which can be called ``geometrical.''  For example, one common definition is that a group acts faithfully and co-compactly on a metric space.   


There is a theorem of Rubin we wish to briefly mention in relation to this context (see \cite{rubin} for details).  One version of the theorem roughly states that when a group $G$ acts [effectively enough] on a [nice enough] space $X$ (for discussion, let us assume $G\leq Homeo(X)$ and the action and space have other qualities as well), then any isomorphism $\theta$ of $G$ induces a homeomorphism $\phi$ of $X$, and indeed, $\theta$ is realised as the topological conjugation $g~\mapsto~g^\phi$  in the larger group $Homeo(X)$.  We explore one way to interpret the ideas regarding ``geometrical actions'' specifically in the context of subgroups of a ``Rubin group'' acting on it's corresponding ``Rubin space."

  As is well known in the community of researchers with interests in the R. Thompson groups, $F<T<V$ are each ``Rubin groups'' with natural associated actions on the spaces $(0,1)$, $S^1$, and the Cantor space $\cc$, respectively.  For those familiar with the theory of $F$ and $T$, we observe the spaces $[0,1]$ and $S^1$  are the quotients of $\cc$ which transform the actions of the subgroups $F$ and $T$ on $\cc$ to Rubin type actions on the quotient spaces (we are playing a bit ``fast and loose'' with the definitions here by adding the points $0$ and $1$ to $(0,1)$; the induced action of $F$ on $[0,1]$ of course fixes these ``extra'' points). See \cite{cfp} for a general survey of $F$, $T$ and $V$.

Taking a cue from the ``co-compact'' aspects from the example of a ``geometrical action'' mentioned above, given a group-space pair $(H,X)$, the authors of \cite{collinolga} define the class $\dot{\mathfrak{D}}_{(H,X)}$ of \emph{demonstrative subgroups of $H$} (for $X$).

\begin{definition}
Suppose $H$ is a group that acts on a space $X$. We say that a subgroup $G \le H$ is a demonstrative subgroup of $H$ over $X$ if there exists a non-empty open subset $U \subset X$ such that for any two elements $g_1,g_2 \in G$ we have $Ug_1 \cap Ug_2 = \emptyset$ if and only if $g_1 \neq g_2$.
\end{definition}

If a subgroup $G$ has this property then we say $G$ is in the set $\dot{\mathfrak{D}}_{(H,X)}$ and all groups isomorphic to $G$ are in the class $\mathfrak{D}_{(H,X)}$ of demonstrable groups for the group-space pair $(H,X)$. We call the open set $U$ a \textit{demonstration} set.

Suppose that $G \le V$ is a demonstrative subgroup of $V$. By Lemma 3.1 of \cite{collinolga} there exists a vertex $n \in \{0,1\}^\ast \subset \mathcal{T}_2$ such that $[n]\cap [n]g = \emptyset$ for all $g \neq 1$. We call $n$ a \emph{demonstration node} for $G$. Note that $n$ will not be the unique demonstration node for $G$, any node that has an address with $n$ as a prefix will also suffice.

{\flushleft {\it Implications:}}\\
In \cite{collinolga} the authors show in Theorem 1.4 that if $D_1,$ $D_2$ are demonstrable groups for $V$, then one can find demonstrative embeddings $\hat{D}_1$, $\hat{D}_2$ of $D_1$ and $D_2$ so that $\langle \hat{D}_1,\hat{D}_2\rangle\cong D_1*D_2$ (while the free product so generated may not be demonstrative).  They also show that if $D$ is demonstrable for $V$, and $A$ is any subgroup of $V$, then the restricted wreath product $A\wr D$ embeds in $V$ (Theorem 1.2 of \cite{collinolga}).  It is this latter embedding result that combines with the main result of this paper to produce our Corollary \ref{cor:fourProps} and Theorem \ref{wreath-embeddings} from the introduction. As Theorem \ref{wreath-embeddings} states that if $G$ is a subgroup of $V$ and $T$ is any countable virtually free group then $G\wr T$ embeds as a subgroup of $V$, we then have that all four of the known closure properties of the $\cocf$ groups hold for the finitely generated subgroups of $V$ (see Corollary \ref{cor:fourProps}).

\section{The Main Results}

\subsection{The countable virtually free groups are in the class $\mathfrak{D}_{(V,\cc)}$}

The goal of this section is to prove that the countable virtually free groups are in the class $\mathfrak{D}_{(V,\cc)}$. We begin by introducing a new way of proving that a subgroup of $V$ is demonstrative which is based on the work done in \cite{collinolga}. The following lemma follows easily from the definitions.

\begin{lemma}
Let $G \le V$. If $[0] \cap [0]g = \emptyset$ for all non-trivial $g \in G$, then $G$ is a demonstrative subgroup of $V$.
\label{demcrit}
\end{lemma}

By Lemma 3.4 in \cite{collinolga} we actually have a stronger result.

\begin{cor}
\label{cor:demnode}
A group $G$ is demonstrable for $V$ if and only if there exists a subgroup $\widetilde{G} \le V$ such that $\widetilde{G} \cong G$ and $[0] \cap [0]g$ for all non-trivial $g \in \widetilde{G}$.
\end{cor}

The rest of the section will be used to show that countable virtually free groups are demonstrable for $V$. This will be done by constructing a demonstrative embedding of the modular group $\Gamma = C_2 \ast C_3 \cong \langle x,y | x^2=y^3=1\rangle$ into $V$. As the class of demonstrative groups is closed under passing to subgroups we will have then proven that any countable free group admits a demonstrative embedding into $V$. In \cite{fryetal} it is shown that the class of demonstrable groups for $V$ is closed under passage to finite index overgroups, which will give us our final result.

We begin by proposing an embedding $\psi:G \rightarrowtail V$, where $G = \langle \alpha, \beta | \alpha^2,\beta^3\rangle \cong \Gamma$ factors as the free product of its subgroups $C_2 = \langle\alpha\rangle$ and $C_3=\langle \beta\rangle$.

We will define the mapping $\psi$ from $G$ to $V$ by determining where the generators $\alpha$ and $\beta$ map to, and quoting von Dyck's theorem.

We now define $\psi(\alpha) \mapsto a$ and $\psi(\beta) \mapsto b$, where $a$ and $b$ are given by the tree-pairs below.

	\begin{figure}[h]
		\centering \def\svgwidth{300pt}
		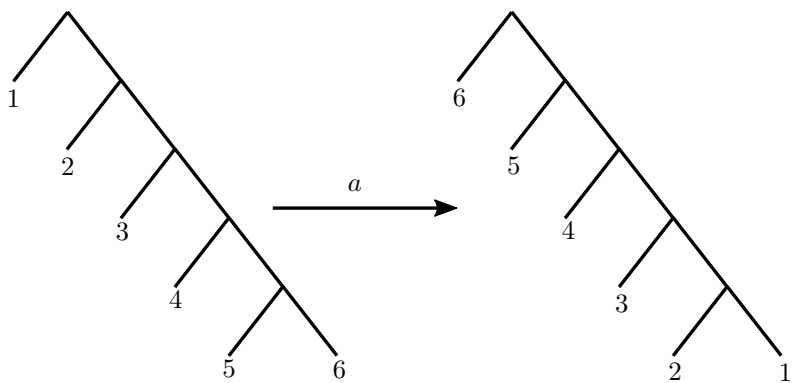
		\caption{The image of $\alpha$ under $\psi$ in Thompson's group $V$}
		\label{defa}
	\end{figure}

	\begin{figure}[h]
		\centering \def\svgwidth{250pt}
		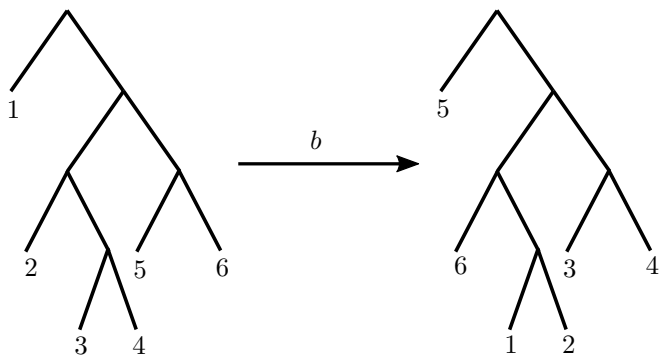
		\caption{The image of $\beta$ under $\psi$ in Thompson's group $V$}
		\label{defb}
	\end{figure}
	
	 An important property to notice is that $\psi(\alpha)=a$ maps the open set $[10]$ to the open set $[11110]$. This dynamical property will be important later when we prove that the homomorphism $\psi:G \to V$ is an embedding.
	
	Similarly, we observe that both $\psi(\beta)=b$ and $(\psi(\beta))^{-1}=b^{-1}$ take the open set $[111]$ and map it into the open set $[10]$. Remembering how the action of $a$ maps $[10]$ into $[11110]$, we begin to see how a ``back and forth" action could be produced by alternating non-trivial elements of $\psi(C_2)$ and $\psi(C_3)$.

  It is immediate that $a$ has order two and $b$ has order three, so by von Dyck's theorem \cite{vonDyck}, the map $\psi$ extends uniquely to a well defined group homomorphism from $G$ to $V$, which we will still call $\psi$, below.

\begin{prop}
\label{C2C3embed}
The group homomorphism $\psi : C_2 \ast C_3 \to V$ induced by $\alpha\mapsto a$ and $\beta\mapsto b$ 
is an embedding.
\end{prop}

To prove Proposition \ref{C2C3embed} we use Fricke and Klein's well known criterion, the Ping-Pong Lemma. The version we give here is based on the one found in \cite{delaharpe}.

\begin{lemma}[Ping-Pong Lemma]
Let $G$ be a group acting on a set $X$ and let $H_1$ and $H_2$ be two subgroups of $G$ such that $|H_1|\ge 3$ and $|H_2| \ge 2$. Suppose there exists two non-empty subsets $X_1$ and $X_2$ such that the following three conditions hold
\begin{enumerate}
	\item $X_2 \not \subset X_1$
	\item for all non-trivial $h_1 \in H_1$, $h_1(X_2) \subset X_1$ 
	\item for all non-trivial $h_2 \in H_2$, $h_2(X_1) \subset X_2$

\end{enumerate}
Then $\langle H_1,H_2 \rangle \cong H_1 \ast H_2$.
\label{pingpong}
\end{lemma}

\begin{lemma}
	The group $G_\psi = \langle a,b \rangle \le V$ factors as $\langle a\rangle \ast \langle b\rangle \cong \Gamma$.
\label{Gembed}
\end{lemma}

\begin{proof}
To be able to use the Ping-Pong Lemma we have to find two sets $X_1$, $X_2$ in $\cc$ that satisfy the three properties given above. Let $X_1 = [10]$ and $X_2 = [111]$ be two open sets in $\cc$. Immediately we see that $X_2 \not \subset X_1$ and so the first condition is met.

We now identify $\langle b \rangle$ and $\langle a \rangle$ with the subgroups $H_1$ and $H_2$ given in lemma \ref{pingpong}. Let $a$ be the non-trivial element in $\langle a \rangle$ as given in Figure \ref{defa}. Observe $(X_1)a=([10])a=[11110] \subset [111] = X_2$. This confirms the third requirement.

For $\langle b \rangle$ we have two non-trivial elements, $b$ and $b^{-1}$. Observe $(X_2)b=([111])b=[100] \subset [10] = X_1$, and similarly for $b^{-1}$, $(X_2)b^{-1}=([111])b^{-1}=[1011] \subset [10] = X_1$. Thus we have shown that all three of the requirements in lemma \ref{pingpong} are met and thus by the Ping-Pong Lemma $G_\psi =\langle a, b \rangle$ factors as $\langle a \rangle\ast \langle b \rangle \cong \Gamma$.
\end{proof}

Before we proceed we define notation that will be used in the future. Suppose $g \in G_\psi$ where $a$ and $b$ are as before. Then $g$ can be written in a unique \textit{normal form} as $g=g_1g_2g_3\ldots g_n$ where $g_i \in \{a,b,b^{-1}\}$ and if $g_j \in \{a\}$ then $g_{j+1} \in \{b,b^{-1}\}$ and vice versa. This normal form is in fact \textit{geodesic}, the shortest path from the identity to the element $g$ in the cayley graph of $G_\psi$. If we have two elements $g,h \in G_\psi$ written in normal form such that $g=g_1g_2g_3 \ldots g_n$ and $h=h_1h_2\ldots h_m$ then $gh = g_1g_2g_3 \ldots g_nh_1h_2\ldots h_m \in G_\psi$ is the concatenation of $g$ and $h$. The element $gh$ is also in normal form if and only if $g_n$ and $h_1$ are not contained within the same subgroup, $\langle a \rangle$ or $\langle b \rangle$, of $G_\psi$. For the rest of the paper, unless stated otherwise, we will assume that all our group elements are written in normal form. We define a function $len:\langle a,b \rangle \to \mathbb{N}_0$ by $len(g)=n$ for any $g=g_1g_2g_3\ldots g_n$ written in the normal form. This is the well defined geodesic length for $g$. We define the length of the identity to be zero. 

\begin{lemma}
\label{backandforth}
Suppose $g \in G_\psi$ and $len(g) \ge 2$. Then
\begin{equation}
[0]g \subseteq
\begin{cases}
    [111],& \text{if } g \text{ ends with generator } a\\
    [10], & \text{if } g \text{ ends with either of the generators } b \text{ or } b^{-1}
\end{cases}
\end{equation}
\end{lemma}

\begin{proof}
Let $\mathcal{P}(g)$ be the statement

\begin{equation*}
[0]g \subseteq
\begin{cases}
    [111],& \text{if } g \text{ ends with generator } a\\
    [10], & \text{if } g \text{ ends with either of the generators } b \text{ or } b^{-1}
\end{cases}
\end{equation*}

for some $g \in G_\psi$ such that $len(g) \ge 2$.

We will proceed by induction on the length of $g$.
Suppose $g \in G_\psi$ such that $len(g) = 2$. There are four options, namely, $ab$, $ab^{-1}$, $ba$ and $b^{-1}a$. Suppose $g=ab$, then $[0]g=[10011]$ and $\mathcal{P}(g)$ holds. Suppose $g=ab^{-1}$, then $[0]g=[101111]$ and $\mathcal{P}(g)$ holds. Suppose $g=ba$, then $[0]g=[1111010]$ and $\mathcal{P}(g)$ holds. Finally suppose $g=b^{-1}a$, then $[0]g=[1110]$ and $\mathcal{P}(g)$ holds.

Suppose $\mathcal{P}(g)$ is true for all $g \in G_\psi$ such that $2 \le len(g) \le n$, $n \in \mathbb{N}$. Now suppose $h\in G_\psi$ such that $len(h)=n+1$. Let $h'$ be the prefix of length $n$ of $h$ when $h$ is written in normal form. 

Suppose $h'$ ends with $a$, there are two options for $h$, either $h=h'b$ or $h=h'b^{-1}$. As $len(h')=n$, by our inductive assumption $[0]h' \subseteq [111]$. Thus if $h=h'b$, then $[0]h = [0]h'b \subseteq [111]b = [100]$ and $\mathcal{P}(h)$ is true. Suppose $h=h'b^{-1}$, then $[0]h = [0]h'b^{-1} \subseteq [111]b^{-1} = [1011]$ and again $\mathcal{P}(h)$ is true.

Suppose instead that $h'$ ends with either $b$ or $b^{-1}$. Then there is only one option for $h$, namely $h=h'a$. As $len(h')=n$, by our inductive assumption $[0]h' \subseteq [10]$. Thus $[0]h = [0]h'a \subseteq [10]a = [11110]$ and $\mathcal{P}(h)$ is true.

Therefore for all $h \in G_\psi$ such that $len(h) = n+1$, the statement $\mathcal{P}(h)$ is true, and therefore by induction $\mathcal{P}(g)$ must be true for all $g \in G_\psi$ such that $len(g) \ge 2$.

\end{proof}

\begin{lemma}
For every non-trivial $g\in G_\psi$,
\begin{center}
	$[0] \cap [0]g = \emptyset$
\end{center}
\end{lemma}

\begin{proof}
Let $\mathcal{Q}(g)$ be the statement $[0] \cap [0]g = \emptyset$, for some $g \in G_\psi$. Suppose $g\in G_\psi$ such that $len(g)=1$. There are three options, namely $a$, $b$ and $b^{-1}$. Suppose $g=a$, then $[0]g=[11111]$ and $\mathcal{Q}(g)$ holds. Suppose $g=b$, then $[0]g=[1010]$ and again $\mathcal{Q}(g)$ holds. Finally suppose $g=b^{-1}$, then $[0]g=[110]$ and $\mathcal{Q}(g)$ holds. 

For all $g \in G_\psi$ such that $len(g) \ge 2$, $\mathcal{Q}(g)$ is true by lemma \ref{backandforth}. Thus $[0] \cap [0]g = \emptyset$ for all non-trivial $g \in G_\psi$.

\end{proof}

Therefore, by lemma \ref{demcrit}, we have the following corollary.

\begin{cor}
The group $G_\psi \cong \Gamma$ is a demonstrative subgroup of $V$
\end{cor}

The inclusion of the countable free groups in the class of demonstrative subgroups of $V$ follows from part of Lemma 3.2 in \cite{collinolga} which is given below.

\begin{lemma}[3.2, Bleak, Salazar-D\'{i}az]
\label{demsub}
Suppose that $G$ is a demonstrative group with $m$ serving as a demonstration node. Then given any subgroup $H \le G$, $H$ is also demonstrative with demonstration node $m$.
\end{lemma}

The free group on two generators $F_2$ is isomorphic to the subgroup $\langle [a,b],[a,b^{-1}] \rangle \le G_\psi$ (where the bracket $[x,y]$ represents the commutator $x^{-1}y^{-1}xy$), and therefore by Lemma \ref{demsub} we have the following corollary.

\begin{cor}
\label{F2dem}
The free group on two generators, $F_2$, is in the class $\mathfrak{D}_{(V,\cc)}$
\end{cor}

Virtually free groups are groups that contain a free group as a finite index subgroup. While it is known (see \cite{roverthesis,collinolga}) that if a group $G$ embeds in $V$, the any finite index over-group of $G$ also embeds into $V$, the paper \cite{fryetal} extends this powerfully with Theorem 3.3, which we paraphrase below.

\textbf{Theorem 3.3} (Berns-Zieve et al) \textit{Suppose $G$ is a group which embeds in R. Thompson's group $V$.  If $G \le H$ where $[H:G]=m$, for some $m \in \mathbb{N}$ and $G$ embeds as a demonstrative subgroup in $V$, then $H$ also embeds as demonstrative subgroup of $V$.}

The theorem tells us that $\mathfrak{D}_{(V,\cc)}$ is closed under taking finite index overgroups. As all countable free groups embed into $F_2$ and since virtually free groups are, by definition, finite index overgroups of free groups, by Corollary \ref{F2dem} countable virtually free groups are contained within $\mathfrak{D}_{(V,\cc)}$.

\subsection{Finitely generated demonstrable groups of $V$ are virtually free}

In this section we prove a partial converse of our previous result.

\begin{lemma}
\label{CFdemonst}
If $G$ is isomorphic to a demonstrative subgroup of Thompson's Group $V$ then $G$ is virtually free.
\end{lemma}

By the remarkable and well-known result of Muller and Schupp \cite{mullerschupp83} \cite{mullerschupp85} we know that the finitely generated virtually free groups are exactly those finitely generated groups that have a context free word problem, the $\mathcal{CF}$-groups. To prove our lemma we must for any given  finitely generated demonstrative group be able to construct a push-down automaton that accepts the word problem of that group. (See the Definition \ref{PDAdef} for the definition of a push-down automaton)

We begin with a motivating example and then generalise our method to encompass all the groups isomorphic to f.g. demonstrative subgroups of $V$. The example below is taken from \cite{collinolga} where they give a demonstrative embedding of $\mathbb{Z}$ generated by the element $g$ given below in Figure \ref{generatorz}.

	\begin{figure}[H]
		\centering \def\svgwidth{250pt}
		\input{Ztreepair.pdf_tex}
		\caption{The generator of a demonstrative copy of $\mathbb{Z}$ inside $V$. The demonstrative node is $n=0$.}
		\label{generatorz}
	\end{figure}
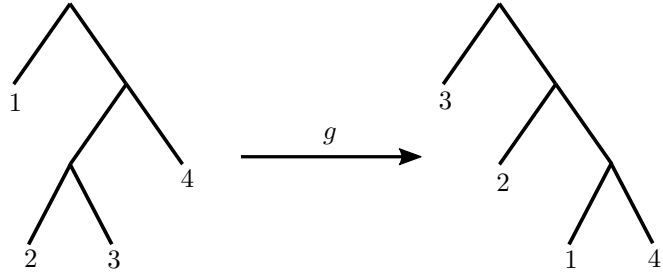
	
Let $G = \langle g \rangle \cong \mathbb{Z}$ be this demonstrative subgroup of $V$. A demonstrative node for $G$ is at the address $n=0 \in \{0,1\}^\ast$. It is already known that $\mathbb{Z}$ is a $\mathcal{CF}$-group so our example gives us no new result, but the method we use to create the push-down automaton that accepts it's word problem can be generalised for every f.g. demonstrative subgroup of $V$.

Let $\mathcal{A}$ be our PDA that accepts the word problem of $G$. $\mathcal{A}$ has three states $\{q_0,q_r,q_a\}$ where $q_0$ is the start state and $q_a$ is the only accept state. Our stack alphabet is the set $\Gamma=\{\#,0,1\}$, where $\#$ is a special bottom-of-the-stack symbol. We read in strings constructed from the generator $g$ and it's inverse $g^{-1}$. The \textit{PDA} $\mathcal{A}$ is defined by the transition table given by Table \ref{TranTable} below.

\begin{table}[H]
\begin{center}
  \begin{tabular}{| c | c | c | c | c |}
    \hline
    Current State & Input & Stack Top & Stack Replacement & New State \\ \hline
    $q_0$ & $\epsilon$ & $\emptyset$ & $0\#$ & $q_a$ \\ \hline
    $q_a$ & $g$ & $0$ & $110$ & $q_r$ \\ \hline
    $q_a$ & $g^{-1}$ & $0$ & $101$ & $q_r$ \\ \hline
    $q_r$ & $g$ & $0$ & $110$ & $q_r$ \\ \hline
    $q_r$ & $g$ & $100$ & $10$ & $q_r$ \\ \hline
    $q_r$ & $g$ & $11$ & $111$ & $q_r$ \\ \hline
    $q_r$ & $g$ & $1010$ & $00$ & $q_r$ \\ \hline
    $q_r$ & $g$ & $1011$ & $01$ & $q_r$ \\ \hline
    $q_r$ & $g^{-1}$ & $0$ & $101$ & $q_r$ \\ \hline
    $q_r$ & $g^{-1}$ & $10$ & $100$ & $q_r$ \\ \hline
    $q_r$ & $g^{-1}$ & $111$ & $11$ & $q_r$ \\ \hline
    $q_r$ & $g^{-1}$ & $1100$ & $00$ & $q_r$ \\ \hline
    $q_r$ & $g^{-1}$ & $1101$ & $01$ & $q_r$ \\ \hline
    $q_r$ & $g$ & $101\#$ & $0\#$ & $q_a$ \\ \hline
    $q_r$ & $g^{-1}$ & $110\#$ & $0\#$ & $q_a$ \\ \hline
  \end{tabular}
\end{center}
\caption{The transition table of the automaton accepting the word problem of $G \cong \mathbb{Z}$}
\label{TranTable}
\end{table}

We also provide a visual representation of $\mathcal{A}$ in Figure \ref{zautomata}.
	
	\begin{figure}[h]
		\centering \def\svgwidth{350pt}
		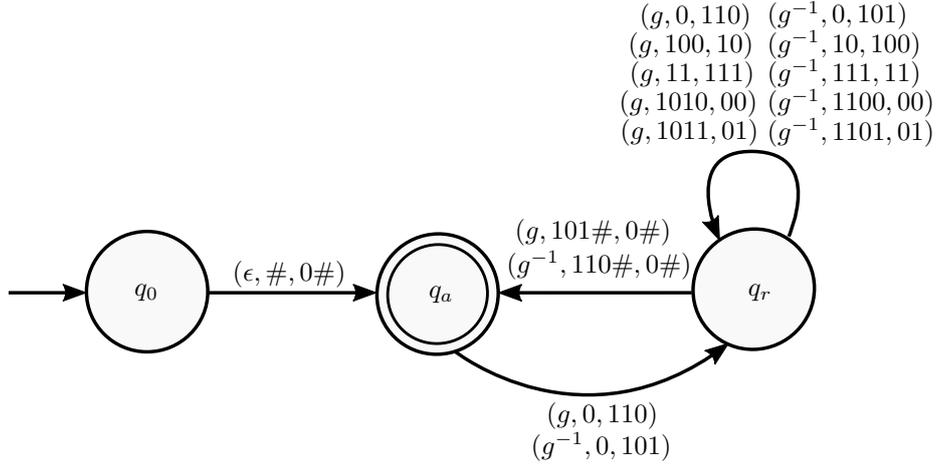
		\caption{A graphical representation of the automata $\mathcal{A}$ that accepts the word problem of $G$.}
		\label{zautomata}
\end{figure}
	
The automata models the action of the generators on the demonstrative node $0$.  The stack will represent the location of the demonstrative node under the action of the word read so far.  Whenever the automaton $\mathcal{A}$ processes a letter, it amends the stack according to the prefix replacement rules defined for the elements of $V$ representing the letters $g$ and $g^{-1}$.  At all times, the top of the stack represents the beginning of the address in $\mathcal{T}_2$ of the node to which the previously processed word has moved the demonstration node $n=0$.

 The automaton $\mathcal{A}$ begins with active state $q_0$ by loading the stack with the address of the demonstrative node $n=0$, and moving the active state to the accept state $q_a$.  Note that none of the input string is read at this time.   Whenever the active state is $q_a$, if $\mathcal{A}$ has finished reading the input then it accepts the word.  However, if the active state is $q_a$ and there are still more letters to be read then $\mathcal{A}$ will process the next letter (which action will move the active state to $q_r$ and modify the stack according to the prefix replacement rules).  From the state $q_r$ there are circumstances which allow the active state to return to $q_a$.  Namely, whenever the active state is $q_r$ and $\mathcal{A}$ processes a letter and the resultant stack is ``$0\#$'', then the active state transitions to $q_a$.  

 By the definition of demonstration nodes, a demonstration node under the action of an element $w$ of the demonstrative group is taken to itself if and only if $w$ is the identity element.  By construction, our automata has stack ``$0\#$'' only when the previously processed word represents the trivial element.  However, this is precisely at the times that the automaton's active state is $q_a$.

We now generalise this method into a proof of Lemma \ref{CFdemonst}.

\begin{proof}

Suppose $\widehat{G}$ is a finitely generated, demonstrable group, isomorphic to a demonstrative subgroup $G$ of Thompson's Group $V$ where $G$ will be generated by elements $\{g_1,g_2,\ldots ,g_m\}$. Suppose  $n \in \{0,1\}^\ast$ is a demonstrative node for $G$ in $\mathcal{T}_2$. Note that by Corollary \ref{cor:demnode} we can always find $G$ such that $n=0$. We describe and construct our automaton $\mathcal{A}$ below.

Let $X = \{g_1,g_2,\ldots,g_m\} \cup \{g_1^{-1},g_2^{-1},\ldots,g_m^{-1}\}$, the union of the set of generators of $G$ and their inverses, be the input alphabet.  Set $\Gamma = \{\#,0,1\}$ to be the stack alphabet. The new automaton $\mathcal{A}$ will also have three states $q_0$, $q_a$ and $q_r$, where $q_a$ is the automaton's only accept state.  We will describe the transitions from each of these states.

\textbf{Transitions from $q_0$} The automaton $\mathcal{A}$ begins in the state $q_0$.  That state admits one transition, which loads the stack with the string $n\# \in \Gamma^\ast$, and transfers active state to the state $q_a$, without reading any of the input.  After this transition, the stack will contain the address of the demonstration node $n$ and the bottom-of-the-stack symbol, with $n$ written from top to bottom on the stack. For example if $n=100$ then $1$ would be at the top of the stack followed by two $0$'s and finally $\#$. We call this the \textit{loading phase}. 

\textbf{Transitions from $q_a$} After the loading phase, $\mathcal{A}$ enters the \textit{reading phase}, where it begins to read the input string from $X^\ast$.  Observe that the current stack is precisely ``$n\#$,'' and this will be true whenever $q_a$ is the active state, by construction.  From $q_a$, there are transitions to the state $q_r,$ defined as follows.  The transitions will be given by tuplets of the form \textit{(input letter, current top-of-stack, top-of-stack re-write)}:

\begin{enumerate}
\item Input letter: $g\in X$.
\item Current top of stack: the string ``$n$''.
\item Stack re-write: The result of applying the prefix replacement determined by the element of $V$ that the symbol $g$ represents, to the string $n$.
\end{enumerate}

(Note: we will only list transitions of $\mathcal{A}$ which can actually arise.  E.g., in our previous example we do not list transitions from $q_a$ with the top-of-stack beginning with a ``$1$''.)  

\textbf{Transitions from $q_r$} All transitions from $q_r$ take the active state to either $q_r$ or to $q_a$.   There is a finite list of pairs $(stack, g)$ for $stack$ representing the full stack, including the $\#$ symbol, and $g$ a letter of our input alphabet, so that the result of applying the prefix replacement determined by the element $\alpha_g\in V$ representing the input letter $g$ to the whole stack is ``$n\#$.''  For such pairs, we add transitions as given by the tuple $(g, stack, n\#)$, which transitions move the active state to $q_a$.  

We now discuss the transitions from $q_r$ to $q_r$.  For each $g$ in our input alphabet, there is a set $\{s_1,s_2,\ldots,s_j\}$ of minimal prefixes which determine the element $\alpha_g$ of $V$ corresponding to the letter $g$ as a prefix replacement map, where we define the corresponding set of strings $\{t_1,t_2,\ldots,t_j\}$ which are the replacement strings, so that $s_i\cdot \alpha_g = t_i$, for $1\leq i\leq j$ indices.  We add transitions given by the tuples $(g,s_i,t_i)$ from $q_r$ to $q_r$.

Note that the non-determinism above allows poor choices that can result in the active state being $q_r$ at the end of reading the input, even though the stack will actually read ``$n\#$,''  however, there will still be a path through the automaton which would have ended at $q_a$ for this input (literally, we could simply change the last choice taken).  In our example, we removed the non-determinism by using the tuples $(g,1010,00),$ $(g,1011,01),$ $(g^{-1},1100,00),$ $(g^{-1},1101,01)$ instead of using the two tuples $(g,101,0), (g^{-1},110,0)$ that our process above produces.  Note that we could force determinism in this general construction by using various carefully selected transitions, as in our example, but we felt our approach here was clearer.

\textbf{Termination.} When we reach the end of the input string, if we are in $q_a,$ then the stack must be $n\#$ by construction, and the input string is equivalent to the identity in our group (by the definition of a demonstrative embedding). If the active state at end of input is $q_r$ then the word is rejected.   If this happens, it means that either the stack is not ``$n\#$,'' and so the element did not act as the identity, or if the stack is ``$n\#$,'' then the automaton made poor choices in the face of non-determinism.  Thus $\mathcal{A}$ accepts the word problem of $G$ and rejects all other strings.

\end{proof}

\bibliographystyle{plain}

\input{FreegpsDemonstrative_arXiv2.bbl}
\end{document}

%% file: treesplitting.pdf_tex
\begingroup%
  \makeatletter%
  \providecommand\color[2][]{%
    \errmessage{(Inkscape) Color is used for the text in Inkscape, but the package 'color.sty' is not loaded}%
    \renewcommand\color[2][]{}%
  }%
  \providecommand\transparent[1]{%
    \errmessage{(Inkscape) Transparency is used (non-zero) for the text in Inkscape, but the package 'transparent.sty' is not loaded}%
    \renewcommand\transparent[1]{}%
  }%
  \providecommand\rotatebox[2]{#2}%
  \ifx\svgwidth\undefined%
    \setlength{\unitlength}{290.89828852bp}%
    \ifx\svgscale\undefined%
      \relax%
    \else%
      \setlength{\unitlength}{\unitlength * \real{\svgscale}}%
    \fi%
  \else%
    \setlength{\unitlength}{\svgwidth}%
  \fi%
  \global\let\svgwidth\undefined%
  \global\let\svgscale\undefined%
  \makeatother%
  \begin{picture}(1,0.69821083)%
    \put(0,0){\includegraphics[width=\unitlength,page=1]{treesplitting.pdf}}%
    \put(0.49035986,0.66686933){\color[rgb]{0,0,0}\makebox(0,0)[lb]{\smash{$\epsilon$}}}%
    \put(0.61814896,0.26961209){\color[rgb]{0,0,0}\makebox(0,0)[lb]{\smash{$k$}}}%
    \put(0.55774499,0.01250675){\color[rgb]{0,0,0}\makebox(0,0)[lb]{\smash{$n_1$}}}%
    \put(0.64147397,0.07823426){\color[rgb]{0,0,0}\makebox(0,0)[lb]{\smash{$n_2$}}}%
    \put(0.48615396,0.21668572){\color[rgb]{0,0,0}\makebox(0,0)[lb]{\smash{$p_{n_1}$}}}%
    \put(0.66672551,0.21668572){\color[rgb]{0,0,0}\makebox(0,0)[lb]{\smash{$p_{n_2}$}}}%
  \end{picture}%
\endgroup%

%% file: Ztreepair.pdf_tex
\begingroup%
  \makeatletter%
  \providecommand\color[2][]{%
    \errmessage{(Inkscape) Color is used for the text in Inkscape, but the package 'color.sty' is not loaded}%
    \renewcommand\color[2][]{}%
  }%
  \providecommand\transparent[1]{%
    \errmessage{(Inkscape) Transparency is used (non-zero) for the text in Inkscape, but the package 'transparent.sty' is not loaded}%
    \renewcommand\transparent[1]{}%
  }%
  \providecommand\rotatebox[2]{#2}%
  \ifx\svgwidth\undefined%
    \setlength{\unitlength}{485.31974732bp}%
    \ifx\svgscale\undefined%
      \relax%
    \else%
      \setlength{\unitlength}{\unitlength * \real{\svgscale}}%
    \fi%
  \else%
    \setlength{\unitlength}{\svgwidth}%
  \fi%
  \global\let\svgwidth\undefined%
  \global\let\svgscale\undefined%
  \makeatother%
  \begin{picture}(1,0.40426236)%
    \put(0,0){\includegraphics[width=\unitlength,page=1]{Ztreepair.pdf}}%
    \put(-0.12693743,-0.12529754){\color[rgb]{0,0,0}\makebox(0,0)[lb]{\smash{}}}%
    \put(-0.00205245,0.24103187){\color[rgb]{0,0,0}\makebox(0,0)[lb]{\smash{$1$}}}%
    \put(0.02096209,0.00539965){\color[rgb]{0,0,0}\makebox(0,0)[lb]{\smash{$2$}}}%
    \put(0.14584707,0.00373445){\color[rgb]{0,0,0}\makebox(0,0)[lb]{\smash{$3$}}}%
    \put(0.2556789,0.12532269){\color[rgb]{0,0,0}\makebox(0,0)[lb]{\smash{$4$}}}%
    \put(0.64205672,0.24344671){\color[rgb]{0,0,0}\makebox(0,0)[lb]{\smash{$3$}}}%
    \put(0.72840931,0.12029378){\color[rgb]{0,0,0}\makebox(0,0)[lb]{\smash{$2$}}}%
    \put(0.83014982,0.00363404){\color[rgb]{0,0,0}\makebox(0,0)[lb]{\smash{$1$}}}%
    \put(0.95483393,0.00536614){\color[rgb]{0,0,0}\makebox(0,0)[lb]{\smash{$4$}}}%
    \put(0.46888287,0.19838763){\color[rgb]{0,0,0}\makebox(0,0)[lb]{\smash{$g$}}}%
  \end{picture}%
\endgroup%

%% file: C2treepair2.pdf_tex
\begingroup%
  \makeatletter%
  \providecommand\color[2][]{%
    \errmessage{(Inkscape) Color is used for the text in Inkscape, but the package 'color.sty' is not loaded}%
    \renewcommand\color[2][]{}%
  }%
  \providecommand\transparent[1]{%
    \errmessage{(Inkscape) Transparency is used (non-zero) for the text in Inkscape, but the package 'transparent.sty' is not loaded}%
    \renewcommand\transparent[1]{}%
  }%
  \providecommand\rotatebox[2]{#2}%
  \ifx\svgwidth\undefined%
    \setlength{\unitlength}{569.82632374bp}%
    \ifx\svgscale\undefined%
      \relax%
    \else%
      \setlength{\unitlength}{\unitlength * \real{\svgscale}}%
    \fi%
  \else%
    \setlength{\unitlength}{\svgwidth}%
  \fi%
  \global\let\svgwidth\undefined%
  \global\let\svgscale\undefined%
  \makeatother%
  \begin{picture}(1,0.46915444)%
    \put(0,0){\includegraphics[width=\unitlength,page=1]{C2treepair2.pdf}}%
    \put(0.4237089,0.24240642){\color[rgb]{0,0,0}\makebox(0,0)[lb]{\smash{$a$}}}%
    \put(-0.00174811,0.34882771){\color[rgb]{0,0,0}\makebox(0,0)[lb]{\smash{$1$}}}%
    \put(0.06607179,0.26295673){\color[rgb]{0,0,0}\makebox(0,0)[lb]{\smash{$2$}}}%
    \put(0.13414491,0.18070168){\color[rgb]{0,0,0}\makebox(0,0)[lb]{\smash{$3$}}}%
    \put(0.20079987,0.09561022){\color[rgb]{0,0,0}\makebox(0,0)[lb]{\smash{$4$}}}%
    \put(0.26745481,0.00910068){\color[rgb]{0,0,0}\makebox(0,0)[lb]{\smash{$5$}}}%
    \put(0.40501925,0.00768252){\color[rgb]{0,0,0}\makebox(0,0)[lb]{\smash{$6$}}}%
    \put(0.55476493,0.34985607){\color[rgb]{0,0,0}\makebox(0,0)[lb]{\smash{$6$}}}%
    \put(0.62258481,0.263985){\color[rgb]{0,0,0}\makebox(0,0)[lb]{\smash{$5$}}}%
    \put(0.69065796,0.18172995){\color[rgb]{0,0,0}\makebox(0,0)[lb]{\smash{$4$}}}%
    \put(0.75731294,0.09102275){\color[rgb]{0,0,0}\makebox(0,0)[lb]{\smash{$3$}}}%
    \put(0.82396787,0.00451321){\color[rgb]{0,0,0}\makebox(0,0)[lb]{\smash{$2$}}}%
    \put(0.96153232,0.00309505){\color[rgb]{0,0,0}\makebox(0,0)[lb]{\smash{$1$}}}%
    \put(0.64920112,0.27644306){\color[rgb]{0,0,0}\makebox(0,0)[lb]{\smash{}}}%
  \end{picture}%
\endgroup%

%% file: C3treepair2.pdf_tex
\begingroup%
  \makeatletter%
  \providecommand\color[2][]{%
    \errmessage{(Inkscape) Color is used for the text in Inkscape, but the package 'color.sty' is not loaded}%
    \renewcommand\color[2][]{}%
  }%
  \providecommand\transparent[1]{%
    \errmessage{(Inkscape) Transparency is used (non-zero) for the text in Inkscape, but the package 'transparent.sty' is not loaded}%
    \renewcommand\transparent[1]{}%
  }%
  \providecommand\rotatebox[2]{#2}%
  \ifx\svgwidth\undefined%
    \setlength{\unitlength}{484.75338913bp}%
    \ifx\svgscale\undefined%
      \relax%
    \else%
      \setlength{\unitlength}{\unitlength * \real{\svgscale}}%
    \fi%
  \else%
    \setlength{\unitlength}{\svgwidth}%
  \fi%
  \global\let\svgwidth\undefined%
  \global\let\svgscale\undefined%
  \makeatother%
  \begin{picture}(1,0.52369562)%
    \put(0,0){\includegraphics[width=\unitlength,page=1]{C3treepair2.pdf}}%
    \put(-0.00205485,0.36035952){\color[rgb]{0,0,0}\makebox(0,0)[lb]{\smash{$1$}}}%
    \put(0.02445087,0.12203431){\color[rgb]{0,0,0}\makebox(0,0)[lb]{\smash{$2$}}}%
    \put(0.09956994,0.00363829){\color[rgb]{0,0,0}\makebox(0,0)[lb]{\smash{$3$}}}%
    \put(0.18625803,0.00363829){\color[rgb]{0,0,0}\makebox(0,0)[lb]{\smash{$4$}}}%
    \put(0.18785809,0.11863313){\color[rgb]{0,0,0}\makebox(0,0)[lb]{\smash{$5$}}}%
    \put(0.31122192,0.12196722){\color[rgb]{0,0,0}\makebox(0,0)[lb]{\smash{$6$}}}%
    \put(0.64150453,0.36200974){\color[rgb]{0,0,0}\makebox(0,0)[lb]{\smash{$5$}}}%
    \put(0.66801024,0.12368453){\color[rgb]{0,0,0}\makebox(0,0)[lb]{\smash{$6$}}}%
    \put(0.74312938,0.00528861){\color[rgb]{0,0,0}\makebox(0,0)[lb]{\smash{$1$}}}%
    \put(0.82981741,0.00528861){\color[rgb]{0,0,0}\makebox(0,0)[lb]{\smash{$2$}}}%
    \put(0.83141747,0.12028335){\color[rgb]{0,0,0}\makebox(0,0)[lb]{\smash{$3$}}}%
    \put(0.95478128,0.12361744){\color[rgb]{0,0,0}\makebox(0,0)[lb]{\smash{$4$}}}%
    \put(0.45255567,0.31318276){\color[rgb]{0,0,0}\makebox(0,0)[lb]{\smash{$b$}}}%
  \end{picture}%
\endgroup%

%% file: zautomata2.pdf_tex
\begingroup%
  \makeatletter%
  \providecommand\color[2][]{%
    \errmessage{(Inkscape) Color is used for the text in Inkscape, but the package 'color.sty' is not loaded}%
    \renewcommand\color[2][]{}%
  }%
  \providecommand\transparent[1]{%
    \errmessage{(Inkscape) Transparency is used (non-zero) for the text in Inkscape, but the package 'transparent.sty' is not loaded}%
    \renewcommand\transparent[1]{}%
  }%
  \providecommand\rotatebox[2]{#2}%
  \ifx\svgwidth\undefined%
    \setlength{\unitlength}{647.62678492bp}%
    \ifx\svgscale\undefined%
      \relax%
    \else%
      \setlength{\unitlength}{\unitlength * \real{\svgscale}}%
    \fi%
  \else%
    \setlength{\unitlength}{\svgwidth}%
  \fi%
  \global\let\svgwidth\undefined%
  \global\let\svgscale\undefined%
  \makeatother%
  \begin{picture}(1,0.48283549)%
    \put(0,0){\includegraphics[width=\unitlength,page=1]{zautomata2.pdf}}%
    \put(0.23950634,0.19057891){\color[rgb]{0,0,0}\makebox(0,0)[lb]{\smash{$(\epsilon,\#,0\#)$}}}%
    \put(0.67499571,0.4675701){\color[rgb]{0,0,0}\makebox(0,0)[lb]{\smash{$(g,0,110)$}}}%
    \put(0.6630398,0.43634446){\color[rgb]{0,0,0}\makebox(0,0)[lb]{\smash{$(g,100,10)$}}}%
    \put(0.66426256,0.40511882){\color[rgb]{0,0,0}\makebox(0,0)[lb]{\smash{$(g,11,111)$}}}%
    \put(0.65303218,0.37389319){\color[rgb]{0,0,0}\makebox(0,0)[lb]{\smash{$(g,1010,00)$}}}%
    \put(0.65427998,0.34266755){\color[rgb]{0,0,0}\makebox(0,0)[lb]{\smash{$(g,1011,01)$}}}%
    \put(0.81158682,0.46875765){\color[rgb]{0,0,0}\makebox(0,0)[lb]{\smash{$(g^{-1},0,101)$}}}%
    \put(0.81158682,0.43723262){\color[rgb]{0,0,0}\makebox(0,0)[lb]{\smash{$(g^{-1},10,100)$}}}%
    \put(0.81158682,0.40570759){\color[rgb]{0,0,0}\makebox(0,0)[lb]{\smash{$(g^{-1},111,11)$}}}%
    \put(0.81158682,0.37418248){\color[rgb]{0,0,0}\makebox(0,0)[lb]{\smash{$(g^{-1},1100,00)$}}}%
    \put(0.81158682,0.34265745){\color[rgb]{0,0,0}\makebox(0,0)[lb]{\smash{$(g^{-1},1101,01)$}}}%
    \put(0.54175978,0.23517932){\color[rgb]{0,0,0}\makebox(0,0)[lb]{\smash{$(g,101\#,0\#)$}}}%
    \put(0.53289964,0.19906776){\color[rgb]{0,0,0}\makebox(0,0)[lb]{\smash{$(g^{-1},110\#,0\#)$}}}%
    \put(0,0){\includegraphics[width=\unitlength,page=2]{zautomata2.pdf}}%
    \put(0.5752702,0.03765571){\color[rgb]{0,0,0}\makebox(0,0)[lb]{\smash{$(g,0,110)$}}}%
    \put(0.55897954,0.00385419){\color[rgb]{0,0,0}\makebox(0,0)[lb]{\smash{$(g^{-1},0,101)$}}}%
    \put(0.13476462,0.17480708){\color[rgb]{0,0,0}\makebox(0,0)[lb]{\smash{$q_0$}}}%
    \put(0.44960951,0.17233652){\color[rgb]{0,0,0}\makebox(0,0)[lb]{\smash{$q_a$}}}%
    \put(0.79128659,0.17480708){\color[rgb]{0,0,0}\makebox(0,0)[lb]{\smash{$q_r$}}}%
  \end{picture}%
\endgroup%